\documentclass[letter]{article}

\usepackage[utf8]{inputenc}
\usepackage[hidelinks]{hyperref}

\usepackage{amsmath, amssymb, amscd, amsthm, amsfonts,gensymb}
\usepackage{mathtools,empheq} 
\usepackage{graphicx,svg}
\usepackage{soul} 
\usepackage{todonotes}
\usepackage{url}
\usepackage{color}
\usepackage{tikz}
\usetikzlibrary{calc}
\usepackage{wrapfig}
\usepackage{subcaption}
\usepackage{caption}

\usepackage[canadian]{babel} 

\usepackage[dvipsnames]{xcolor}
\usepackage{pdfcomment}
\usepackage{enumitem}
\usepackage{cleveref}
\usepackage{wrapfig}
\usepackage{array}
\usepackage{import}
\usepackage{pdfpages}
\usepackage{transparent}

\graphicspath{ {figures/} }
\newtheorem{theorem}{Theorem}

\newtheorem{lemma}[theorem]{Lemma}

\newtheorem{conjecture}[theorem]{Conjecture}

\newtheorem{corollary}[theorem]{Corollary}

\newtheorem{claim}[theorem]{Claim}

\theoremstyle{remark}
\newtheorem{remark}{Remark}

\theoremstyle{definition}

\DeclareMathOperator{\conv}{conv}
\DeclareMathOperator{\rec}{rec}
\DeclareMathOperator{\verts}{vert}
\DeclareMathOperator{\lin}{lin}
\DeclareMathOperator{\aff}{aff}
\DeclareMathOperator{\bd}{bd}
\DeclareMathOperator{\inter}{int}

\DeclareMathOperator{\pos}{pos}

\DeclareMathOperator{\Ill}{Ill}
\DeclareMathOperator{\relint}{relint}

\usepackage[canadian]{babel} 
\usepackage[
backend=biber,
dateabbrev=false,
style=numeric,
giveninits=true,
maxnames=99,
url=true,
doi=true,
]{biblatex}
\AtEveryBibitem{\clearfield{day}}
\AtEveryBibitem{\clearfield{month}}

\DeclareFieldFormat[article]{title}{\mkbibemph{#1}}
\DeclareFieldFormat[online]{title}{\mkbibemph{#1}}
\DeclareFieldFormat[inproceedings]{title}{\mkbibemph{#1}}
\DeclareFieldFormat{titlecase}{#1}

\begin{document}

\title{Illuminating Primitive Polytopes}

\author{Illya Ivanov}

\maketitle

\begin{abstract}

A convex $d$-dimensional polytope may be defined as a bounded intersection of closed halfspaces in $\mathbb{E}^d$ with interior. 
A polytope is primitive if omitting any halfspace renders the intersection unbounded. 
In this paper, we prove the Illumination Conjecture for the primitive polytopes: any primitive convex $d$-polytope $P$ can be illuminated with at most $2^d$ directions; even fewer if $P$ is not a linear image of a $d$-cube. 

\end{abstract}

\section{Introduction}

\subsection{Definitions}

We denote the $d$-dimensional Euclidean space by $\mathbb{E}^d$, and its origin point by $\mathbf{o}$. 
The notation $\left<\cdot, \cdot \right>$ stands for the standard Euclidean product, and we denote the corresponding norm by $\left| \cdot \right| $. 
The Euclidean distance between points $\mathbf{a}, \mathbf{b} \in \mathbb{E}^d$ is denoted by $\left| \mathbf{a}, \mathbf{b} \right|$ . 
The origin-centred unit sphere in $\mathbb{E}^d$ is denoted by $\mathbb{S}^{d-1}$.  
The closed Euclidean ball in $\mathbb{E}^d$ with centre $\mathbf{x}$ and radius $r$ is denoted by $\mathbf{B}_{r}(\mathbf{x})$.
 The closed ray, starting at a point $\mathbf{x} \in \mathbb{E}^d$ with direction $\mathbf{u} \in \mathbb{S}^{d-1} $ is denoted by $R_{\mathbf{u}}(\mathbf{x})$. The line that contains a point $\mathbf{x}$, and is parallel to a direction $\mathbf{u}$, is denoted by $L_{\mathbf{u}}\left( \mathbf{x} \right) $.  The  segment that connects distinct points $\mathbf{a}, \mathbf{b} \in \mathbb{E}^d$ is denoted by $\left[ \mathbf{a}, \mathbf{b} \right] $. 

For some $\alpha\ge 0$ and $\mathbf{u} \in \mathbb{E}^d \setminus \left\{ \mathbf{o} \right\} $, the \textit{hyperplane}\index{hyperplane} $H_{\mathbf{u},\alpha}$ is the set $\left\{ \mathbf{x} \in \mathbb{E}^d : \left< \mathbf{x},\mathbf{u} \right> = \alpha \right\} $. The \textit{closed negative halfspace}\index{negative halfspace} $H^{-}_{\mathbf{u}, \alpha}$
(respectively, the \textit{closed positive halfspace} $H^{+}_{\mathbf{u}, \alpha}$) is the set $\left\{ \mathbf{x} \in \mathbb{E}^d : \left<\mathbf{x},\mathbf{u} \right> \le \alpha\right\} $ $\left( \text{respectively, }\left\{ \mathbf{x} \in \mathbb{E}^d : \left<\mathbf{x},\mathbf{u} \right> \ge \alpha\right\} \right)$. Unless specified otherwise, every mentioned halfspace is closed.

A linear combination $\mathbf{u} = \mu_1 \mathbf{x}_1 + \mu_2 \mathbf{x}_2 + \ldots + \mu_n\mathbf{x}_n$ is \textit{affine}\index{affine combination} if $\sum_{i=1}^{n} \mu_n = 1$.  A \textit{positive combination}\index{positive combination} of points is a linear combination with non-negative coefficients. A linear combination of points is \textit{convex}\index{convex combination}, if it is both affine and positive.
The \textit{linear, affine, positive or convex hull}\index{convex hull}\index{linear hull}\index{affine hull} of $S \subset \mathbb{E}^d$ is, respectively, the set of all the linear, affine, positive, or convex combinations of the points in $S$.
The linear, affine, positive or convex hull of $S \subset \mathbb{E}^d$ is denoted, respectively, by $ \lin S, \aff S, \pos S$, and $ \conv S$.

A set $K \subset \mathbb{E}^d$ is a \textit{convex body} in $\mathbb{E}^d$ if $K$ is convex, closed, bounded, and has a nonempty interior. 
The space of convex bodies in $\mathbb{E}^d$ is denoted $\mathcal{K}^d$.
If $K \in  \mathcal{K}^d$ and $S \subset \bd K$, then the set of all points $\mathbf{x} \in S$, for which there exists a small $\varepsilon$ such that $\left( \mathbf{B}_{\varepsilon}\left( \mathbf{x} \right) \cap \bd K \right) \subset S$, is the \textit{relative interior}\index{relative interior} of $S$.
The relative interior of $S$ is denoted by $\relint S$.

An \textit{affine subspace}\index{affine subspace} $S \subset \mathbb{E}^d$ is a linear subspace of $\mathbb{E}^d$ translated by some vector. The \textit{dimension}\index{dimension of a set} of an affine subspace is the dimension of its generating linear subspace.
The dimension of a set $S \subset \mathbb{E}^d$, denoted by $\text{dim }S$, is the dimension of $\text{ aff }S $.
An affine subspace $S \subset \mathbb{E}^d$ such that $\text{ dim } S = k < d-1$ is called a $k$\textit{-flat}.
An affine subspace $L \subset \mathbb{E}^d$ \textit{supports (is a support subspace of)} a convex body $K \in \mathcal{K}^{d}$ at a point $\mathbf{x} \in  \bd K$  if $\mathbf{x} \in  L$, but $L$ has an empty intersection with $\inter K$. 
In particular, we use the notion of \textit{support hyperplane}\index{support hyperplane}.
If the convex body $K \in  \mathcal{K}^d$ lies in a halfspace $H^{-}_{\mathbf{u}, \alpha}$, and the hyperplane $H_{\mathbf{u},\alpha}$ supports $K$, then $H^{-}_{\mathbf{u}, \alpha}$ is a \textit{support halfspace}\index{support halfspace} of $K$.
If, in addition, $K \subset H^-_{\mathbf{u},\alpha}$, then $\mathbf{u}$ is an \textit{outer normal}\index{outer normal} of $K$ at $\mathbf{x}$.
A \textit{face}\index{face} of $K$ is the intersection of $K$ and some supporting hyperplane of $K$.

Sets $A, B \subset \mathbb{E}^d$ are \textit{separated by}\index{separation}\index{separating hyperplane} a hyperplane $H_{\mathbf{u},\alpha}$ if $A \subset H^{-}_{\mathbf{u},\alpha}, B \subset H^{+}_{\mathbf{u},\alpha}$, or vice versa. If,
in addition,
$A \cap H_{\mathbf{u},\alpha} = B \cap H_{\mathbf{u},\alpha} = \emptyset$,
then $H_{\mathbf{u},\alpha}$ \textit{strictly separates}\index{strict separation}  $A$ and $B$.
The \textit{polar set}\index{polar set} of $S \subset \mathbb{E}^d$ is the set of points $S^{o} = \left\{ \mathbf{x} \in \mathbb{E}^d : \left<\mathbf{x}, \mathbf{y} \right> \le 1 \text{ for all } \mathbf{y} \in S\right\} $. If $K$ is a convex body, and $F \subset \bd K$ is a face of $K$, then the \textit{face of $K^{o}$ polar to $F$} is the set $F^{*} = \left\{ \mathbf{x} \in \bd K^{o} : \left<\mathbf{y}, \mathbf{x} \right> = 1 \text{ for all } \mathbf{y} \in F\right\} $. If a convex body $K$ has a face $F$, then the \textit{polar face of}\index{polar face} $F$, denoted by $F^{*}$, is the following set:

\begin{equation}
	F^{*} = \left\{ \mathbf{x} \in \bd F^{*} : \left<\mathbf{x}, \mathbf{y} \right> = 1 \text{ for any } \mathbf{y} \in F \right\}.
\end{equation}

A \textit{convex polyhedral set}\index{convex polyhedral set} is a reduced intersection of some finite number of halfspaces in $\mathbb{E}^d$, that has an interior. 
The facet of a convex polyhedral set $P$, corresponding to the halfspace $H^{-}_i$, is the intersection of $P$ with the corresponding hyperplane.
A $d$-dimensional \textit{convex polytope}\index{convex polytope} is a convex hull of a finite number of points in $\mathbb{E}^d$ that has an interior.  A convex body $P \in \mathcal{K}^d$ is a convex $d$-polytope if and only if it is a bounded convex polyhedral set with interior. 
A \textit{bounding hyperplane} of a polyhedral set $P$ is the affine hull of some facet of $P$.  We denote the set of bounding hyperplanes of a polyhedral set $P$ by $\mathcal{H}(P) = \left\{ \aff F : F \text{ is a facet of }P\right\}$, and we call it the \textit{bounding set} of $P$. 

\begin{remark}
	In this paper, we don't study the non-convex polytopes, so whenever we mention a polytope, it is convex. Same goes for the polyhedral sets.
\end{remark}

A point $\mathbf{x}$ in $\bd K, K \in \mathcal{K}^d$ is \textit{illuminated}\index{illuminated point} by a direction $\mathbf{u} \in \mathbb{S}^{d-1}$ if the ray $R_{\mathbf{u}}\left( \mathbf{x} \right) $ intersects $\inter K$.
A set of directions \textit{illuminates} $K$ if every point in $\bd K$ is illuminated by at least one direction from the set.
The smallest number of directions that illuminate $K$ is the \textit{illumination number}\index{illumination number} of $K$.
We denote the illumination number of a convex body $K \in \mathcal{K}^d$ by $\Ill(K)$.

\subsection{Illumination Conjecture}

We can now write down one of at least five equivalent statements of the Illumination Conjecture itself, stated by V. Boltyanskiy \cite{boltyanskiyProblemIlluminatingBoundary1960}:

\begin{conjecture}
A convex body $K \in \mathcal{K}^d$ can be illuminated by at most $2^{d}$ directions, even fewer if $K$ is not a linear image of a $d$-dimensional cube.
\end{conjecture}

Prior to Boltyanskiy, the equivalent versions of the Illumination Conjecture were stated by H. Hadwiger \cite{hadwigerUngelosteProblemeNr1957, hadwigerUngelosteProblemeNr1960}, and A. Markus together with I. Gohberg \cite{gohbergCertainProblemCovering1960}.

The Illumination Conjecture is conclusively proven for $K \in \mathcal{K}^{d}$ only in case $d = 2$ (see, for example,  Boltyanskiy and  Gohberg's \cite{boltyanskiyResultsProblemsCombinatorial1985} book). 
The best available three-dimensional estimate, $\Ill (K) \le 14: K \in \mathcal{K}^{3}$, is by A. Prymak \cite{prymakNewBoundHadwigers2023a}.
As of February 2026, the best general upper bound is by M. Campos, P. van Hintum, R. Morris, and M. Tiba \cite{camposHadwigersConjectureBourgain2024}:
\begin{equation}
Ill_d \le 4^d \exp \left( -c \left( \frac{d}{(\log d)^8} \right) \right) \text{ for some universal constant } c.
\end{equation}

The problem of illuminating convex polytopes is wide open even in $\mathbb{E}^{3}$. 
The Illumination Conjecture is proven for some particular classes of polytopes, such as 3-polytopes with affine symmetry  (K. Bezdek \cite{bezdekProblemIlluminationBoundary1991}), dual cyclic polytopes (K. Bezdek, T. Bisztriczky \cite{bezdekProofHadwigersCovering1997}), totally-sewn 4-polytopes (T. Bizstriczky, F. Fodor \cite{bisztriczkySeparationTheoremTotallysewn2015}), strongly monotypic polytopes (V. Bui \cite{buiCharacterizationStronglyMonotypic2021}), and zonotopes (H. Martini \cite{martiniResultsProblemsZonotopes1985}). 
This list is by no means exhaustive, for the detailed history of the Illumination Conjecture, and the related results, refer to the survey articles by K. Bezdek and M. Khan \cite{bezdekGeometryHomotheticCovering2018}, G. Fejes T\'{o}th \cite{fejestothFourClassicProblems2023}, and H. Martini \cite{martiniCombinatorialProblemsIllumination1999}.

\subsection{Primitive Polytopes and Positive Bases}

If, for some convex polytopes $P, P' \in \mathcal{K}^d: P \subsetneq P'$,  the bounding set $\mathcal{H}(P')$ is a strict subset of $\mathcal{H}(P)$, then $P'$ is a \textit{spill}\index{spill} \textit{of} $P$ \textit{through the hyperplanes} $\mathcal{H}(P) \setminus \mathcal{H}(P')$, or through the corresponding facets of $P$.
A convex polytope $P$ is \textit{primitive}, if every spill of $P$ is unbounded.\index{primitive polytope}
Primitive polytopes were first described by E. Steinitz \cite{steinitzUberDiejenigenKonvexen1909} and later studied by L. Fejes T\'{o}th \cite{fejestothPrimitivePolyhedra1962}, and P. McMullen, R. Schneider, and G. Shephard \cite{mcmullenMonotypicPolytopesTheir1974}.  

Our main result is the following theorem:

\begin{theorem}\label{thm:fleanill}
If $P \subset \mathbb{E}^d$ is a primitive polytope, then ${\text Ill} P \leq 2^d$, with equality attained if and only if $P$ is an image of a  $d$-dimensional cube under some invertible linear transformation.
\end{theorem}

 \begin{remark}\label{rmk:tinker}
Our motivation to study the primitive polytopes arose from tinkering. To illuminate a convex polytope, it is necessary and sufficient to illuminate its vertices. That's why, at first, we tried to describe the effect that removing a vertex of a convex polytope has on the illumination number. However, removing vertices one by one inevitably ends on a simplex, and the illumination number of a simplex in $\mathbb{E}^d$ is only $d+1$. Moreover, removing a vertex may increase the illumination number, or may decrease it, and we aimed for a somewhat monotonous process, some transformation that, hopefully, wouldn't decrease the illumination number.
A convex polytope is the convex hull of a finite set of points with interior. However, it may also be defined as a bounded intersection of a finite number of halfspaces with interior. Omitting halfspaces immediately looks more promising than removing vertices: you end on a primitive polytope, and there's several combinatorially different kinds in $\mathbb{E}^d, d \ge 2$. Moreover, a cube is primitive. We proved the Illumination Conjecture for the primitive polytopes. If we could find some process, which takes any convex polytope to some primitive polytope without decreasing the illumination number, then the Illumination Conjecture for the convex polytopes would be proven. Unfortunately, "omitting halfspaces" also may increase the illumination number, or may decrease it, there's no monotonous effect on the illumination number.

 \end{remark}

Let $P$ be the bounded intersection of some finite number of negative halfspaces: $P = H^{-}_1 \cap \ldots \cap H^{-}_n$.  For each $i: 1 \le i \le n$, we denote the respective facet $P \cap H_i$ by $F_i$. Let $\left\{ \mathbf{b}_i  :  1 \le i \le n\right\}$  be the set of outer normals of $P$'s facets. We scale each $b_i$ so that $H_i = H_{\mathbf{b}_i, 1}$. In this way, this set of outer normals is also the set of vertices of the polar polytope $Q = P^{o}$.

\begin{remark}
	Unless specified otherwise, every convex polytope and every polyhedral set contains the origin in its interior.
\end{remark}

We say a vector set $\mathcal{P} = \left\{ \mathbf{b}_1, \ldots, \mathbf{b}_n \right\} $ \textit{positively spans} $\mathbb{E}^{d}$, if $\pos \mathcal{P} = \mathbb{E}^{d}$. A vector set $\mathcal{P}$ is \textit{positively independent} if $\mathbf{b}_i \notin \pos \left\{ \mathcal{P} \setminus \mathbf{b}_i \right\}$ for any $i: 1 \le i \le n$. The vector set $\mathcal{P}$ is a \textit{positive basis} in $\mathbb{E}^{d}$ if $\mathcal{P}$ positively spans $\mathbb{E}^{d}$, and is positively independent.  A positive basis of $\mathbb{E}^{d}$ is \textit{minimal} if it has exactly $d+1$ vectors. A positive basis is \textit{maximal} if it has $2d$ vectors, which is only possible if the positive basis consists of $d$ linearly independent pairs of antipodal vectors.  Given a positive basis $\mathcal{P}$ of $\mathbb{E}^{d}$, a \textit{spanned subspace} of $\mathbb{E}^{d}$ is a linear subspace $S \subset \mathbb{E}^{d}$ such that some proper subset of $\mathcal{P}$ is a positive basis for $S$. If this proper subset is a minimal basis of $S$, then $S$ is a \textit{minimal subspace}.  For more information on the positive bases, see C. Davis' \cite{davisTheoryPositiveLinear1954} paper, where he introduced the concept. Positive bases and primitive polytopes are closely connected:

\begin{theorem}\label{thm:facleanposbas}[McMullen, Schneider, Shephard \cite{mcmullenMonotypicPolytopesTheir1974}, Theorem 3, parts 1,3]
	 A convex polytope $P \subset \mathbb{E}^d$ is primitive if and only if the set of outer normals to the facets of $P$ forms a positive basis in $\mathbb{E}^d$.
\end{theorem}

\begin{corollary}\label{cor:flean_crit}
	A  convex polytope $P \subset \mathbb{E}^d, \mathbf{o} \in \inter P$ is primitive if and only if the vertices of the polar polytope $P^{o}$ form a positive basis in $\mathbb{E}^d$.
\end{corollary}

\subsection{Two Kinds of Primitive Polytopes}

As mentioned before, any convex polytope contains $\mathbf{o}$ in its interior, unless specified otherwise.  Below are the properties of the polar faces of convex polytopes that we use in this chapter.  Polarity reverses polyhedral face inclusion: if $F_1 \subset  F_2$ holds for the faces $F_1, F_2$ of a convex polytope $P$, then $F_2^{*} \subset F_1^{*}$ holds for the respective polar faces of $P^{o}$.
The sum of dimensions of a face and its polar face does not depend on the choice of the face: $\text{ dim }F + \text{ dim }F^{*} = d-1$.
In particular, the polar face of a vertex of $P$ is a facet of $P^{o}$, and vice versa.

\begin{remark}\label{rmk:fleancriterion}
	Initially, we hypothesized that for every non-primitive polytope $P \subset  \mathbb{E}^d$, there is a sequence of polytopes $P = P'_0 \subset P'_1 \subset P'_2 \subset  \ldots \subset  P'_l$, such that:
	\begin{enumerate}
		\item Every $P'_i, 1 \le i \le l$ is a spill of $P'_{i-1}$ through a single facet.
			
		\item For every $i: 1 \le i \le l$, inequality $\Ill(P'_i) \ge \Ill(P'_{i-1})$ holds.
		\item $P'_l$ is primitive.
	\end{enumerate}

	If this hypothesis held, proving the Illumination Conjecture for the primitive polytopes would directly imply the Illumination Conjecture for all the convex polytopes. Unfortunately, this hypothesis is false.

\end{remark}

\begin{claim}
	There exists a convex polytope $P \in  \mathcal{K}^3$, such that any spill of $P$ through a single facet has greater illumination number than  $P$.
\end{claim}

\begin{proof}
	\begin{figure}[h]
		\centering
\def\svgwidth{0.5 \columnwidth}
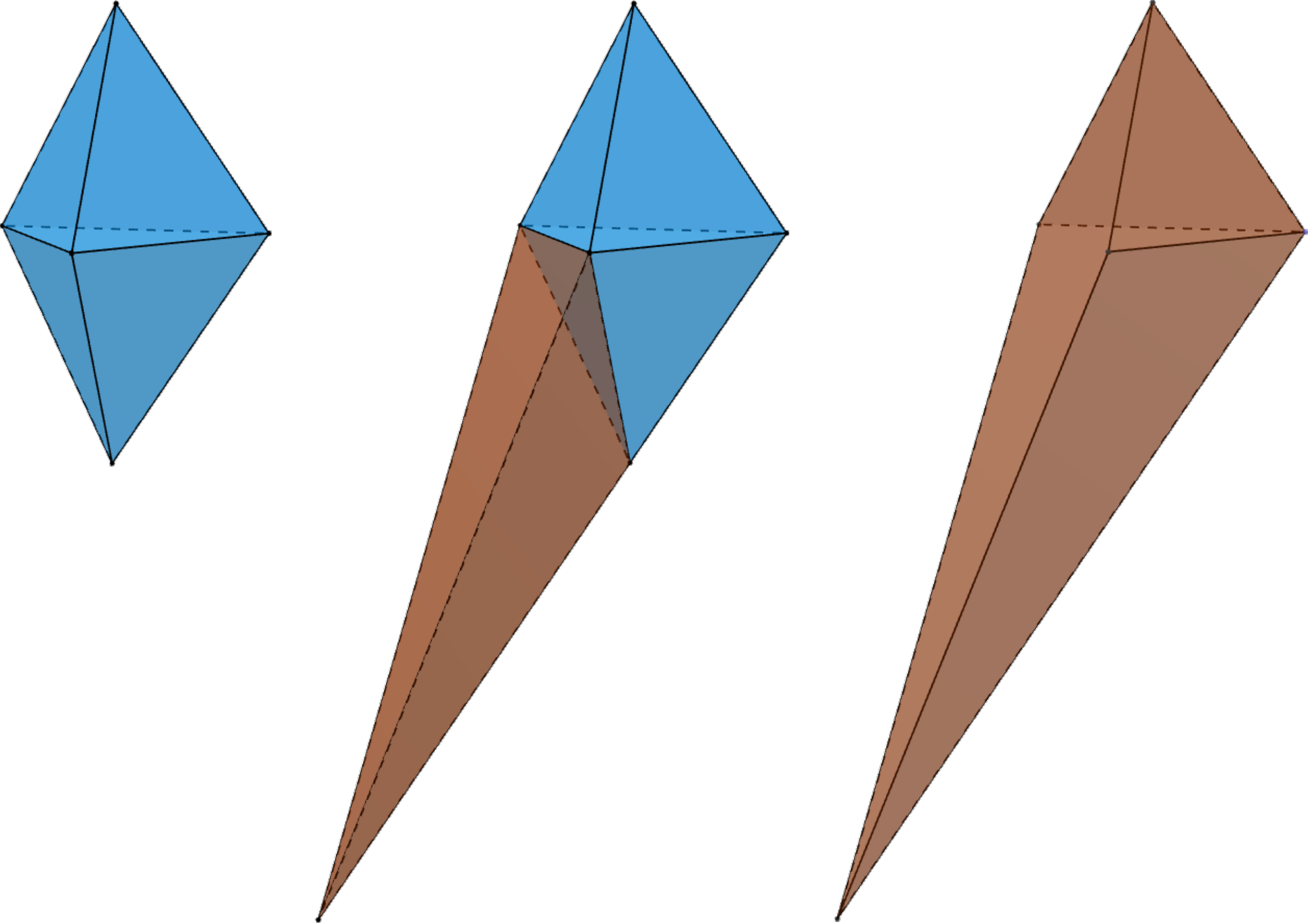

		\caption{Polytope $P$ and its spill through a single facet.}
		\label{fig:fleanbipyr}
	\end{figure}

	Let $P$ be the union of two regular 3-simplices that share a facet  (Figure \ref{fig:fleanbipyr}). No two vertices of $P$ can be illuminated by the same direction, so $\Ill(P) = 5$. For any pair of facets $F_1, F_2 \subset  \bd P$, there is a symmetry of $P$ that takes $F_1$ to $F_2$.  Therefore, all the spills of $P$ through a single facet are congruent. Namely, the spill of $P$ through a single facet is the pyramid  over a quadrilateral, that is not a parallelogram. Thus, $\Ill(P') = 4$.

\end{proof}

The main result of this chapter is Theorem \ref{thm:fleanill}, the Illumination Conjecture for the primitive polytopes. Moreover, we pose the following conjecture:

\begin{conjecture}\label{conj:fleanverts}
	A primitive polytope in $\mathbb{E}^d, d\ge 4$ has at most $2^d$ vertices, even fewer if it isn't an affine $d$-cube.
\end{conjecture}
\begin{remark}
	Conjecture \ref{conj:fleanverts} is trivial in $\mathbb{E}^{2}$, since there are only two classes of primitive polygons: triangles and parallelograms. Primitive polytopes in $\mathbb{E}^{3}$ were enumerated by  Steinitz \cite{steinitzUberDiejenigenKonvexen1909}. Later, L. Fejes T\'{o}th \cite{fejestothPrimitivePolyhedra1962} came up with a shorter proof. 
\end{remark}

\begin{remark}
	A number of vertices of a primitive polytope $P \subset \mathbb{E}^d$ is equal to the number of facets of $\conv \mathcal{P}$, a convex hull of a corresponding positive basis. A positive basis of $\mathbb{E}^d$ can have at most $2d$ vectors. According to McMullen's \cite{mcmullenUpperboundConjectureConvex1971} upper bound theorem, for a given dimension $d\ge 2$, the polytope with the greatest number of facets is  the cyclic polytope.  The number of facets of a cyclic polytope with $n$ vertices in $\mathbb{E}^d$, denoted by $f_{d-1}\left( C(d,n) \right) $ is expressed by the following formula (see, for example, K. Fukuda's \cite{fukudaFrequentlyAskedQuestions2004} survey article):

	\begin{equation}\label{eqn:uppboundfacet}
		f_{d-1}(C(d,n))= \binom{n-\left\lfloor \frac{d+1}{2} \right\rfloor}{n-d} + \binom{n-\left\lfloor \frac{d+2}{2} \right\rfloor}{n-d}
	\end{equation}

	If we plug in $n = 2d$ into the Formula \ref{eqn:uppboundfacet}, use Stirling's approximation $n! \sim \sqrt{2\pi n}\left( \frac{n}{e} \right)^n$ and evaluate equation \ref{eqn:uppboundfacet}, order of magnitude of the result is approximately $2.5^{d}$.
	Therefore, the upper bound theorem does not directly imply that any primitive polytope in $\mathbb{E}^d$ has fewer than $2^{d}$ vertices.
\end{remark}

\subsection{Preliminary Claims}

	A  \textit{convex cone}\index{convex cone}
	$C \subset \mathbb{E}^d$ is a convex set of rays that start at the origin. A restriction of a convex cone to $\mathbb{S}^{d-1}$ is a \textit{spherical cone}\index{spherical cone}. For a convex set $S \subset \mathbb{E}^d$, the \textit{recession cone} of $S$, denoted by $\rec S$, is the set of all the rays $R_{\mathbf{u}}\left( \mathbf{o} \right), \mathbf{u} \in \mathbb{S}^{d-1}$ such that for
any $\mathbf{y} \in S$, the ray $R_{\mathbf{u}}\left( \mathbf{y} \right) $ is contained in $S$.  A \textit{normal cone} of a convex body $K$ at a boundary subset $S \subset  \bd K$ is the set of all the outer normals to $K$ at the points in $S$. We denote the normal cone of $K$ at $S$ by $S^{\perp}(K)$. If $S = \left\{ \mathbf{x} \right\} $ consists of a single point, then we omit the curly brackets, and denote the corresponding normal cone by $\mathbf{x}^{\perp}(K)$. When there is no ambiguity about the convex body, we will write $S^{\perp}$. 
If $K$ is a convex polytope, and $\mathbf{x}$ is a point in the relative interior of a face $F \subset  \bd K$, then $\dim(F) + \dim \mathbf{x}^{\perp}(K) = d$. The set of vertices of a convex polytope $P$ is denoted $\verts P$. A primitive polytope $P$ is \textit{cylindrical} if the corresponding positive basis $\mathcal{P} = \verts \left(   P^{o}\right)$ has a spanned $(d-1)$-subspace.
\index{cylindrical primitive polytope}

\begin{claim}\label{clm:fleanframe}
	Let $F_1 \subset \bd P$ be a facet of a primitive polytope $P \subset \mathbb{E}^d$, and let $H_1$ be the corresponding support hyperplane of $P$. Moreover, let $F_2, \ldots, F_k$ be the facets of $P$ that share a $(d-2)$-face with $F_1$. Consider the set $Q_1 =  H^{-}_2 \cap \ldots \cap H^{-}_k$. Let $Q_1^{-} = Q_1 \cap H_1^{-}$ and $Q_1^{+} = Q_1 \cap H^{+}_1$ (see Figure \ref{fig:fleanpmframe}).
	If $Q_1^{+}$ is unbounded, then the recession cone $\rec Q$ is a line.
	
\end{claim}
\begin{proof}

	Since $P$ is primitive, $Q_1^-$ is unbounded, and contains at least one ray. 
	Pick an arbitrary point $\mathbf{r} \in \relint F_1$. 
	If a closed convex set $S$ contains the ray $R_{\mathbf{u}}\left(\mathbf{x} \right) $ for some $\mathbf{x} \in  S, \mathbf{u} \in  \mathbb{S}^{d-1}$, then, for any $\mathbf{y} \in S$, the ray $R_{\mathbf{u}}\left( \mathbf{y} \right) $ also lies in $S$ (see, for example, Schneider's \cite{schneiderConvexBodiesBrunn2013} book).  Therefore, $Q_1^+$ contains at least one ray that originates in $\mathbf{r}$, let this ray be $R_{\mathbf{u}}\left( \mathbf{r} \right), \mathbf{u} \in \mathbb{S}^{d-1}$. Since $F_1 = Q_1^+ \cap H_1$ is bounded, the point $\mathbf{r} + \mathbf{u}$ lies outside the hyperplane $H_1$.

	\begin{figure}[h]
		\centering
\def\svgwidth{0.25 \columnwidth}
\begingroup%
  \makeatletter%
  \providecommand\color[2][]{%
    \errmessage{(Inkscape) Color is used for the text in Inkscape, but the package 'color.sty' is not loaded}%
    \renewcommand\color[2][]{}%
  }%
  \providecommand\transparent[1]{%
    \errmessage{(Inkscape) Transparency is used (non-zero) for the text in Inkscape, but the package 'transparent.sty' is not loaded}%
    \renewcommand\transparent[1]{}%
  }%
  \providecommand\rotatebox[2]{#2}%
  \newcommand*\fsize{\dimexpr\f@size pt\relax}%
  \newcommand*\lineheight[1]{\fontsize{\fsize}{#1\fsize}\selectfont}%
  \ifx\svgwidth\undefined%
    \setlength{\unitlength}{413.6338698bp}%
    \ifx\svgscale\undefined%
      \relax%
    \else%
      \setlength{\unitlength}{\unitlength * \real{\svgscale}}%
    \fi%
  \else%
    \setlength{\unitlength}{\svgwidth}%
  \fi%
  \global\let\svgwidth\undefined%
  \global\let\svgscale\undefined%
  \makeatother%
  \begin{picture}(1,1.73564749)%
    \lineheight{1}%
    \setlength\tabcolsep{0pt}%
    \put(0,0){\includegraphics[width=\unitlength,page=1]{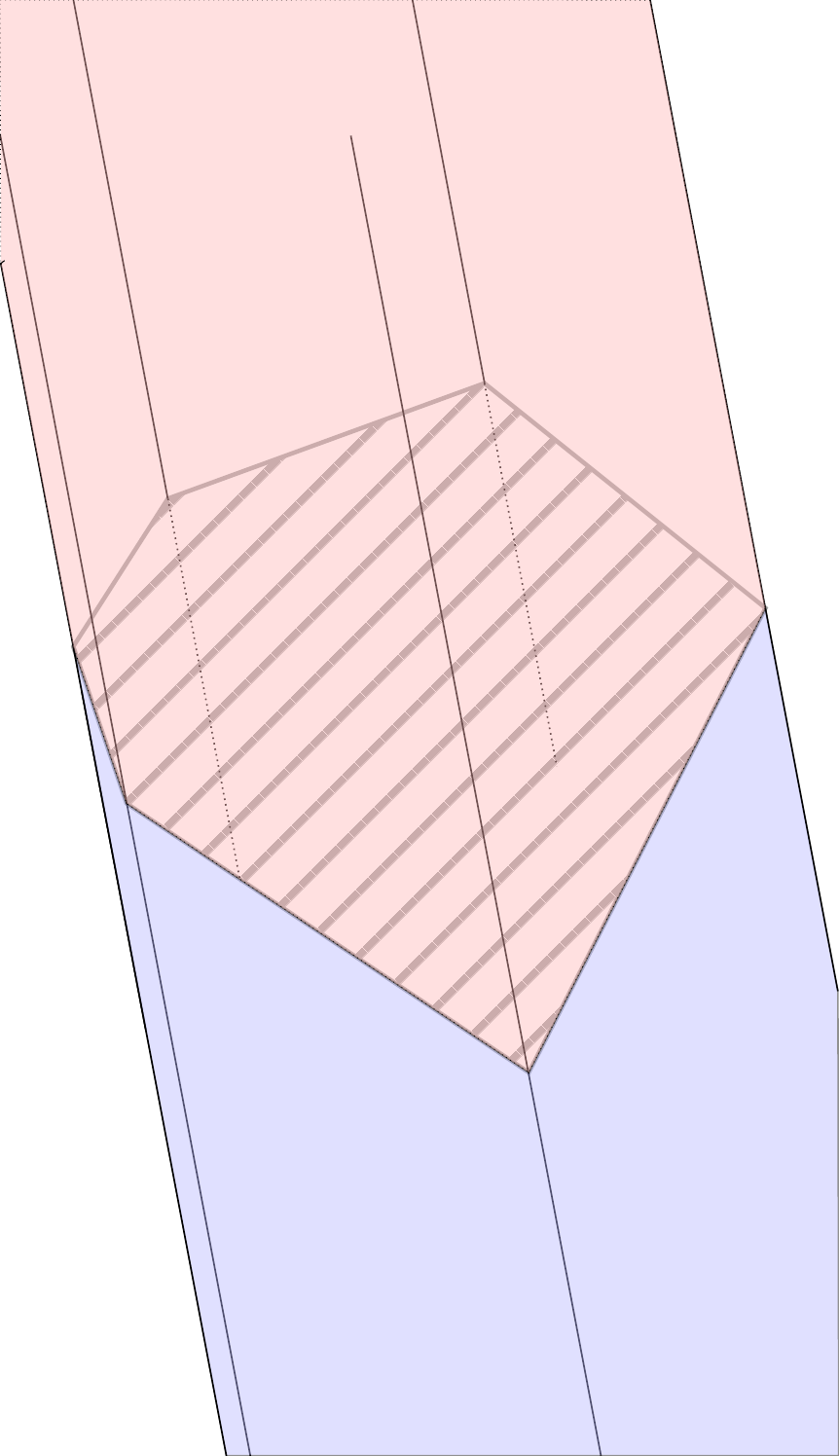}}%
    \put(0.60185706,1.40046246){\color[rgb]{0,0,0}\makebox(0,0)[lt]{\lineheight{1.25}\smash{\begin{tabular}[t]{l}$Q_1^{+}$\end{tabular}}}}%
    \put(0.72805448,0.4700983){\color[rgb]{0,0,0}\makebox(0,0)[lt]{\lineheight{1.25}\smash{\begin{tabular}[t]{l}$Q_1^{-}$\end{tabular}}}}%
    \put(0.41435983,0.91825123){\color[rgb]{0,0,0}\makebox(0,0)[lt]{\lineheight{1.25}\smash{\begin{tabular}[t]{l}$\mathbf{r}$\end{tabular}}}}%
    \put(0.77419114,0.68333583){\color[rgb]{0,0,0}\makebox(0,0)[lt]{\lineheight{1.25}\smash{\begin{tabular}[t]{l}$F_1$\end{tabular}}}}%
    \put(0,0){\includegraphics[width=\unitlength,page=2]{fleanpmframe.pdf}}%
    \put(0.1565796,1.57758291){\color[rgb]{0,0,0}\makebox(0,0)[lt]{\lineheight{1.25}\smash{\begin{tabular}[t]{l}$\mathbf{r}+\lambda\mathbf{b}$\end{tabular}}}}%
    \put(0.37466972,0.27515919){\color[rgb]{0,0,0}\makebox(0,0)[lt]{\lineheight{1.25}\smash{\begin{tabular}[t]{l}$\mathbf{r}+\mu\mathbf{v}$\end{tabular}}}}%
  \end{picture}%
\endgroup%

		\caption{Facet $F_1$ and the corresponding $Q_1$.}
		\label{fig:fleanpmframe}
	\end{figure}

	The set $Q_1^-$ is also unbounded, therefore it contains at least one ray starting in $\mathbf{r}$. Let this ray be $R_{\mathbf{v}}\left( \mathbf{r} \right) $ for some $\mathbf{v} \in \mathbb{S}^{d-1}$.  Since $\mathbf{r} + \mathbf{u} \in \inter H_1^+$ and $\mathbf{r} + \mathbf{v} \in \inter H_1^-$, there is a nontrivial positive combination $\mathbf{c} = \alpha\mathbf{u} + \beta\mathbf{v}, \alpha,\beta \ge 0, \alpha^2 + \beta^2 \neq 0$ such that $\mathbf{r} + \mathbf{c} \in H_1$. Without loss of generality, we can assume $\alpha + \beta = 1$.  For an arbitrary positive coefficient $\gamma$, the vector $r+\gamma \mathbf{c}$ lies in $H_1$. On the other hand, $\mathbf{r} + \gamma\mathbf{c} = \alpha(\mathbf{r} + \gamma \mathbf{u}) + \beta\left( \mathbf{r} + \gamma \mathbf{v} \right)$. It is a convex combination of points in $Q_1$, therefore $\mathbf{r} + \gamma \mathbf{c}$ lies in $Q_1$. We showed that $\mathbf{r} + \gamma \mathbf{c} \in H_1 \cap Q_1 = F_1$ for any positive $\gamma$. However, $F_1$ is bounded, so $\mathbf{c} = \mathbf{o}$. Therefore, any ray in in $Q_1^-$ has direction $-\mathbf{u}$, and any ray in $Q_1^{+}$ has direction $\mathbf{u}$. We showed that $\rec Q$ is a line. 

\end{proof}

Now we are ready to state the following important property of primitive polytopes. We say that the distinct facets $F_i$ and $F_j$ of a convex polyhedron are $k$-adjacent, if $\dim \left( F_i \cap F_j \right) = k $.
\index{$k$-adjacent facets}

\begin{corollary}\label{cor:flean_cyladj}
	Every two facets  of a  non-cylindrical primitive polytope are $(d-2)$-adjacent.
\end{corollary}

\begin{proof}
	The notation in this proof is the same as it is in Claim \ref{clm:fleanframe}.  Suppose a non-cylindrical primitive polytope $P \subset  \mathbb{E}^d$ has $n \ge d+1$ facets. We pick an arbitrary facet, and denote it by $F_1$. 
	We denote all the facets of $P$ that are $(d-2)$-adjacent to $F_1$ by $\left\{ F_2, \ldots, F_k \right\}$. We will prove that $k = n-1$ and, therefore, $F_1$ shares a $(d-2)$-face with every other facet of $P$.

By Claim \ref{clm:fleanframe}, the set $Q_1^{-}= H_1^{-} \cap H_2^- \cap \dots \cap H_k^-$  is bounded. Therefore, $Q_1^{-} = P$, and $F_1$ is $(d-2)$-adjacent to every other facet of $P$. We showed that an arbitrary facet of $P$ is $(d-2)$-adjacent to every other facet of $P$, and that concludes the proof.

\end{proof}

\begin{remark}
A polytope $P \subset \mathbb{E}^d$ is $m$-neighbourly if every $m$ vertices of $P$ form a face of $P$. If $P$ is a non-cylindrical primitive polytope, then Corollary \ref{cor:flean_cyladj} implies that the polar polytope $P^{o}$ is 2-neighbourly.

\end{remark}

\begin{claim}\label{clm:flean_2facet}
	Let $P = H^{-}_1 \cap \ldots \cap H^{-}_n$ be a non-cylindrical primitive polytope, and let $P'_2$ be the spill of $P$ through the facet $F_2 \subset \bd P$.  Then the facet $F'_1 = P'_2 \cap H_1$ of $P'_2$ is unbounded.

\end{claim}

\begin{proof}
	We will demonstrate that there exists a vector $\mathbf{x}$, parallel to $H_1$, such that inequality $\left<\mathbf{x}, \mathbf{b}_i \right> \le 0$ holds for any $i: 3\le i\le n$.  This is a necessary and sufficient condition for $F'_1$ to contain a ray with direction $\mathbf{x}$.  Consider the convex cone $C = \pos \left( \mathbf{b}_3 \cup \ldots \cup \mathbf{b}_n \right) $, and let $C'$ be the projection of $C$ onto $H_1$ along the direction $\mathbf{b}_1$. We will show that $C' \neq H_1$, which is trivial if $\dim C < d-1$.  Suppose $\dim C = d-1$.  The polytope $P$ is non-cylindrical, and $C$ is not a spanned $(d-1)$-subspace of $\mathcal{P}$, so $C \neq \lin C$.  Therefore, $C$ is contained in some halfspace of $\lin C$, and the projection of a $\left(d-1 \right)$-halfspace onto $H_1$ is not the entire hyperplane.

	 Finally, suppose $\dim C = d$.
	 We know $\mathbf{b}_1 \not\in C$, as $\mathcal{P}$
	 is positively independent. Moreover, $\mathcal{P} \setminus \left\{ \mathbf{b}_2 \right\}$ doesn't positively span $\mathbb{E}^d$, so  $\mathbf{b}_1 \not\in \inter\left( -C \right) $.  Therefore, there is a hyperplane $H_C$ that supports $C$ and contains the points $\mathbf{o}, \mathbf{b}_1$. The projection of $H_C$ onto $H_1$ along $\mathbf{b}_1$ supports $C'$, therefore, $C' \neq H_1$.
The convex cone $C'$ is the positive hull of the projections of $\mathbf{b}_3, \ldots, \mathbf{b}_n$ onto $H_1$ along $\mathbf{b}_1$. We showed that $C' \neq H_1$, therefore, $C'$ is contained within some closed halfspace of $H_1$. If $\mathbf{x}$ is the outer normal of that halfspace in $H_1$, then inequality $\left<\mathbf{x}, \mathbf{b}_i \right> \le 0$ holds for any $i: 3\le i\le n$. Therefore, the facet $F'_1$ contains a ray with direction $\mathbf{x}$. That concludes the proof.

\end{proof}

\begin{corollary}\label{cor:fleanfacets}
	Every facet of a non-cylindrical primitive polytope is a primitive $(d-1)$-polytope.
\end{corollary}

Let $Q$ be the polytope, polar to a primitive polytope $P$.
For any $i: 1 \le i\le n$, we denote the polar body of the spill $P'_i$ by $Q'_i = \left( P'_i \right)^{o}$. Equivalently, $Q'_i = \conv\left\{  \mathcal{P}  \setminus \left\{   \mathbf{b}_i\right\}  \right\} $.
Since the set $\verts Q = \mathcal{P}$ is a positive basis, we know $ \mathbf{o} \in \inter Q$, but $\mathbf{o} \not\in \inter Q'_i$.  For any $i: 1 \le i \le n$, we denote the dimension of $\rec \left( P'_i \right) $ by $r_i$. Next, we denote the dimension of the face of $Q'_i$, that contains $\mathbf{o}$ in its relative interior, by $f_i$. If for some $i$ the polytope  $Q'_i$ does not contain the origin, then $f_i = 0$.

\begin{claim}\label{clm:flean_rayaltydim}
	Let $P \subset  \mathbb{E}^d$ be a primitive polytope with $n \ge d+1$ facets. Then $r_i + f_i = d$ for any $i: 1 \le i\le n$.
\end{claim}

\begin{proof}

	Let $P'_i, 1\le i \le n$ be the spill of $P$ through the facet $F_i$, and let $\mathbf{u} \in \mathbb{S}^{d-1}$ be an arbitrary direction.
	The ray $R_\mathbf{u}(\mathbf{o})$ lies in $\rec P'_i$ if and only if $\left<\mathbf{u}, \mathbf{x} \right> \le 0 $ for any $\mathbf{x} \in Q'_i = P^{o}_i$.  Suppose  $\mathbf{o} \not\in  Q'_i$. In this case, $f_i = 0$.  There is a point $\mathbf{r} \in \mathbb{S}^{d-1}$, such that $\left<\mathbf{x}, \mathbf{r} \right> < 0$ for any $\mathbf{x} \in Q'_i$. Therefore, there exists some small $\varphi$, such that  $\left|\widehat{\mathbf{r}, \mathbf{r}'}\right| < \varphi $ for any $\mathbf{r}' \in C\left(\mathbf{r},\varphi\right)$. Then $C\left(\mathbf{r},\varphi\right) \subset \rec P_i'$, and $r_i = d = r_i + f_i $.

Now, suppose $\mathbf{o} \in Q'_i$. Let $G'_i$ be the face of $Q'_i$, that contains $\mathbf{o}$ in its relative interior.  A direction  $\mathbf{r} \in \mathbb{S}^{d-1}$ lies in $\rec P'_i$ if and only if $\mathbf{r}$ is an outer normal of $Q'_i$ at $\mathbf{o}$. That means $r_i = \dim \mathbf{o}^{\perp}(Q'_i)$, and, therefore, $r_i + f_i = \dim \mathbf{o}^{\perp}(Q'_i) + \dim G'_i = d$. That concludes the proof.
\end{proof}

\subsection{Proof of Theorem \ref{thm:fleanill}}

Now we can prove the Illumination Conjecture for the primitive polytopes.
The proof is inductive: the only primitives in $\mathbb{E}^{2}$ are triangles and parallelograms, their illumination numbers are 3 and 4 respectively.
Suppose that the Illumination Conjecture holds for the primitive polytopes in $\mathbb{E}^k: k \in \left\{1, \ldots, d-1 \right\} $ for some fixed $d \ge 3$.

We will use the separation lemma, proved by P. Soltan and V. Soltan \cite{soltanXrayingConvexBodies1986} and, independently, for illumination by affine subspaces, by K. Bezdek \cite{bezdekProblemIlluminationBoundary1991}:

\begin{lemma}\label{lem:polarill}
 Let $\mathbf{v}$ be a boundary point of a convex body $K \in \mathcal{K}^{d}, \mathbf{o} \in \inter K$, and let $F_\mathbf{v}$ be the smallest face of $K$ containing $\mathbf{v}$. A direction $\mathbf{u} \in \mathbb{S}^{d-1}$ illuminates $\mathbf{v}$ if and only if the hyperplane $H_{\mathbf{v},0}$
strictly separates $F_\mathbf{v}^{o}$ from $\mathbf{o}$.
\end{lemma}

First, we consider the case when every facet of $P$ has recession dimension 1.
 \begin{claim} If $P = H^{-}_1 \cap \ldots \cap H^{-}_n$ is a primitive $d$-polytope, and  $\rec F_i = 1$ for $i: 1\le i \le n$, then $P$ is an affine $d$-cube.
\end{claim}
\begin{proof}

	We use the polar polyhedron $Q = P^{o} = \conv\left\{ \mathbf{b}_1, \ldots, \mathbf{b}_n \right\} $, where each $\mathbf{b}_i, 1\le i \le n$ is the vertex of $Q$, polar to the facet $F_i$ of $P$.  If a set of vertices of $Q$ spans a $\left( d-1 \right)$-subspace, then exactly two vertices of $Q$ lie outside that subspace, and are strictly separated by it. By Claim \ref{clm:flean_rayaltydim}, $f_i = d-1$ for any $i: 1\le i \le n$. For any $i: 1\le i \le n$ the point $\mathbf{b}_i$ is strictly separated by some spanned $\left( d-1 \right) $-subspace from some other vertex of $Q$.  Therefore, $n$ is even, and we can index the vertices of $Q$ so that for every $k: 1\le k \le \frac{n}{2}$, the vertices $\mathbf{b}_{2k-1}, \mathbf{b}_{2k} $ are strictly separated by the affine hull of all the other vertices of $Q$. Let $G_k, 1\le k \le \frac{n}{2}$ be the convex hull of the points $\left( \verts Q  \right) \setminus \left\{ \mathbf{b}_{2k-1}, \mathbf{b}_{2k} \right\}$.

	We know that $\mathbf{o} \in  \relint G_k$ for any $k: 1\le k \le \frac{n}{2}$. 
	We pick an arbitrary integer $j: 1\le j \le \frac{n}{2}$. The origin point lies in $\relint \left\{ \cap G_k  : 1\le k \le \frac{n}{2}, k \neq j   \right\} $. 
	However, since every pair $ \left\{ \mathbf{b}_{2k-1}, \mathbf{b}_{2k}\right\}$ is strictly separated by the hyperplane $\aff G_k$, the set $\left\{ \cap G_j  : 1\le k \le \frac{n}{2}, k \neq j \right\} $ is the segment $\left[ \mathbf{b}_{2j-1}, \mathbf{b}_{2j} \right] $. Polytope $Q$ is the convex hull of $\frac{n}{2}$ segments, and each of these segments contains $\mathbf{o}$ in its relative interior. Hence, $P = Q^{o}$ is an affine cube.

\end{proof}

We showed that the only primitive polyhedron $P \in \mathcal{K}^d$, that has all its facets' recession dimensions equal to $1$, is an affine cube. Now, suppose $P$ has at least one facet with recession dimension over 1. 

We pick the facet $F_i \subset \bd P$ with the greatest recession dimension $r_i > 1$. Now, suppose $r_i > 1$. We look at $Q'_i = (P'_i)^{o} = \conv\left( \mathcal{P} \setminus \mathbf{b}_i \right) $, the polar set of $P$'s spill through $F_i$. Let $G \subset \bd Q'_i$ be the face of $Q'_i$ that contains $\mathbf{o}$ in its relative interior. If $\mathbf{o} \not\in Q'_i$, then $G = \emptyset$. The set $\verts G$ is a proper subset of $\mathcal{P}$, so it is positively independent. The face $G$ contains $\mathbf{o}$ in its relative interior, therefore, set $\verts G$ is a positive basis of the spanned subspace $S = \lin(G)$. Dimension of $S$ is equal to $f_i = d-r_i$. By inductive assumption, all facets of $G \subset S$ can be separated from $\mathbf{o}$ with at most $2^{f_i}$ $f_i$-flats through the origin.  Since $\mathbf{o} \not\in \inter Q'_i$, the set $Q'_i$ can be separated from the origin $\mathbf{o}$ by some hyperplane $H'_i$ that contains $\mathbf{o}$.

If an $f_i$-flat $J \subset S$ separates a facet of $G$ from $\mathbf{o}$ in $S$, then we can pick a hyperplane in $\mathbb{E}^d$ that contains $J$ and separates all the vertices in $\verts Q_i' \setminus \verts G $.
A set of at most $2^{f_i}$ hyperplanes like that separates all the facets of $Q$ that do not contain $v_i$. Therefore, the respective normal directions illuminate all vertices in $\verts P \setminus F_i$. That means $\Ill(P) \le 2^{f_i}+\Ill_{d-1}\left( F_i \right) \le 2^{f_i} + 2^{d-1} < 2^{d}$. That concludes the proof of Theorem \ref{thm:fleanill}.


\printbibliography

\bigskip

\noindent Illya Ivanov\\
\small{Department of Mathematics and Statistics, University of Calgary, Canada}\\
\small{E-mail: \texttt{illya.ivanov@tutanota.com}}

\end{document}